\documentclass[12pt]{amsart}
\usepackage{amssymb}
\usepackage{verbatim}
\usepackage[usenames]{color}
\usepackage{hyperref}
\usepackage{url}
\usepackage[all]{xy}

\newtheorem{theorem}{Theorem}[section]
\newtheorem{proposition}[theorem]{Proposition}

\newtheorem{maintheorem}[theorem]{Main Theorem}
\newtheorem{cor}[theorem]{Corollary}

\theoremstyle{definition}
\newtheorem{definition}[theorem]{Definition}

\theoremstyle{remark}

\newcommand{\zf}{\textrm{ZF}}
\newcommand{\ch}{\textnormal{CH}}
\newcommand{\m}{-}

\newcommand{\dom}{\textnormal{Dom}}

\newcommand{\cf}{\textnormal{cf}}

\newcommand{\mf}{\mathfrak}
\newcommand{\mc}{\mathcal}
\newcommand{\mbb}{\mathbb}

\newcommand{\p}{\mathcal{P}}

\newcommand{\add}{\textnormal{add}}

\title[A Lower Bound for Generalized Dominating Numbers]
 {A Lower Bound for \\ Generalized Dominating Numbers}
\author{Dan Hathaway}
\address{Mathematics Department\\
University of Michigan\\
Ann Arbor, MI 48109--1043, U.S.A.}
\email{danhath@umich.edu}
\thanks{A portion of the results of this paper were proven during the September 2012 Fields Institute Workshop on Forcing while D.H. was supported by the Fields Institute. }

\begin{document}

\begin{abstract}
We show a new proof
 for the fact that when
 $\kappa$ and $\lambda$ are infinite cardinals
 satisfying $\lambda ^ \kappa = \lambda$,
 the cofinality of the set of all
 functions from $\lambda$ to $\kappa$ ordered by
 everywhere domination is $2^\lambda$.
An earlier proof was a
 consequence of a result
 about independent families of functions.
The new proof follows directly
 from the main theorem we present:
 for every $A \subseteq \lambda$
 there is a function $f: {^\kappa \lambda} \to \kappa$
 such that whenever $M$ is a transitive model of $\zf$
 such that ${^\kappa \lambda} \subseteq M$
 and some $g: {^\kappa \lambda} \to \kappa$ in $M$ dominates $f$,
 then $A \in M$.
That is,
 ``constructibility can be reduced to domination''.
\end{abstract}

\maketitle

\section{Introduction}

Given a partially ordered set $\mbb{P} = \langle P, \le_P \rangle$,
 let $\cf(\mbb{P})$ be the cofinality of $\mbb{P}$.
That is,
 $\cf(\mbb{P}) = \min \{ |A| : A \subseteq P$ and
 $(\forall p \in P)(\exists a \in A)\, p \le_P a \}$.
Let $\kappa$ and $\lambda$ be infinite cardinals.
Regarding $\kappa$ as a partially ordered set
 $\langle \kappa, \le \rangle$,
 it is natural to wonder about the structure
 of the product of this ordering with itself
 $\lambda$ many times.
This is the same as the partially
 ordered set $\langle {^\lambda \kappa} , \le \rangle$
 of all functions from $\lambda$ to $\kappa$,
 ordered by $$
 (\forall f,g \in {^\lambda \kappa})\,\big{[} f \le g \iff
 (\forall \alpha < \lambda) f(\alpha) \le g(\alpha)\big{]}.$$
This is referred to as the
 $\textit{everywhere domination}$ ordering.
When $f \le g$, for brevity we will just say $g$ dominates $f$.
We will often write ${^\lambda \kappa}$ instead of
 $\langle {^\lambda \kappa}, \le \rangle$ when
 no confusion should arise.

Of course, we could be ``more general'' and consider
 the everywhere domination ordering of
 all functions from an arbitrary set $X$ to $\kappa$,
 but that is isomorphic
 to $\langle {^{|X|} \kappa, \le} \rangle$.
Nevertheless, it will sometimes be more
 convenient notationally to have the domains be
 arbitrary sets rather than
 cardinals, so we will do this freely.

We may want to investigate the cofinality
 $\cf({^\lambda \kappa})$ of
 the partially ordered set ${^\lambda \kappa}$.
Without loss of generality, we can take
 $\kappa$ to be regular.
When $\kappa$ is regular and $> \lambda$,
 it is clear that
 $\cf({^\lambda \kappa}) = \kappa$
 (because the set of constant functions is cofinal).
Thus, we might as well assume
 $\kappa \le \lambda$.
When this happens, of course
 $\cf( {^\lambda \kappa} ) \le
 |{^\lambda \kappa}| = 2^\lambda$.


There is also the partial ordering
 $\langle {^\lambda \kappa}, \le^* \rangle$
 of all functions from $\lambda$ to $\kappa$
 by \textit{eventual domination},
 where $f \le^* g$ iff
 $$ (\exists \alpha < \lambda)
    (\forall \beta \ge \alpha)\,
    f(\beta) \le g(\beta).$$
In the literature, this is often investigated more
 than everywhere domination, so we will spend a few
 moments to explain how the two notions are connected.
Of course, $$\cf \langle {^\lambda \kappa}, \le^* \rangle
 \le \cf \langle {^\lambda \kappa}, \le \rangle.$$
A straightforward diagonalization shows
 $$\lambda^+ \le \cf \langle {^\lambda \kappa}, \le^* \rangle.$$
Notice that for a regular $\lambda$,
 $$\cf \langle {^\lambda \lambda}, \le^* \rangle =
 \cf \langle {^\lambda \lambda}, \le \rangle,$$
 because if we take a set $\mc{F}$ cofinal in
 $\langle {^\lambda \lambda}, \le^* \rangle$ and replace each
 $f \in \mc{F}$ with the set of functions of the form
 $$\alpha \mapsto \max \{ f(\alpha), \beta \}$$
 for some $\beta < \lambda$,
 we get a set cofinal in
 $\langle {^\lambda \lambda}, \le \rangle$
 of size $|\mc{F}|$.
However, the same trick \textit{cannot} be used to argue
 $\cf \langle {^\lambda \kappa}, \le^* \rangle =
  \cf \langle {^\lambda \kappa}, \le \rangle$
 when $\lambda$ is a cardinal and $\kappa$
 is some regular cardinal $< \lambda$.
There are other relationships between everywhere
 domination and eventual domination, for example
 when $\lambda$ is regular,
 $$\cf \langle {^\lambda \kappa}, \le \rangle =
 \cf \langle {^\lambda \kappa}, \le^* \rangle \cdot
 \sum_{\alpha < \lambda}
 \cf \langle {^\alpha \kappa}, \le \rangle.$$
Also, for any $\lambda$,
 $$\cf \langle {^\lambda \kappa}, \le \rangle
 \le \cf \langle {^{\lambda^+} \kappa}, \le^* \rangle$$
 (more generally, for any $\mu \ge \lambda^+$,
 $\cf \langle {^\lambda \kappa}, {\le} \rangle \le
  \cf \langle {^\mu \kappa}, {\le^*} \rangle$).
Putting the above two relationships together,
 $$\cf \langle {^{\lambda^+} \kappa}, \le \rangle =
 \cf \langle {^{\lambda^+} \kappa}, \le^* \rangle.$$
For results about alternative notions of eventual domination,
 see \cite{Monk}.

The case where $\kappa = \lambda$
 is investigated in
 \cite{Cummings}
 by Cummings and Shelah.
They show, as part of a more general result,
 that for a regular $\lambda$
 satisfying $\lambda^{<\lambda} = \lambda$,
 there is a
 $\lambda$-closed and $\lambda^+$-c.c.\ 
 forcing which forces
 $\cf( {^\lambda \lambda} ) < 2^\lambda$.

In \cite{Szymanski},
 Szyma\'nski
 determines whether
 $\cf({^\lambda \kappa}) = 2^\lambda$
 in all cases where $\kappa$ is regular,
 $\kappa \le \lambda < 2^\omega$,
 and $2^\omega$ is real-valued measurable.
Specifically,
 he shows that under these assumptions,
 $\cf( {^\omega \omega} ) < 2^\omega$,
 $\cf( {^\lambda \omega} ) = 2^\lambda$
 when $\omega < \lambda$, and
 $\cf( {^\lambda \kappa} ) < 2^\lambda$
 when $\omega < \kappa \le \lambda$.

In \cite{Jech1},
 Jech and Prikry
 essentially show that whenever
 $I$ is a $\kappa^+$-complete ideal on $\lambda$
 and if there is a family $\mc{F}$ of
 pairwise $I$-disjoint functions from
 $\lambda$ to $\kappa$, then
 $|\mc{F}| \le \cf({^\lambda \kappa})$.
They then show various situations in which
 one might have such a family.
They also begin to investigate the case when
 $2^\omega$ is real-valued measurable.

A slightly different approach than using
 a family of
 $I$-almost disjoint functions for some ideal $I$
 is to use a family of
 sufficiently independent functions.
We illustrate this connection in the next
 section, and show the consequence that
 if $\lambda^\kappa = \lambda$,
 then $\cf({^\lambda \kappa}) = 2^\lambda$.
In the section after that, we present our
 main theorem which helps to clarify
 the nature of the domination relation
 of functions from $\lambda$ to $\kappa$
 when $\lambda^\kappa = \lambda$.

\section{Independent Families of Functions}

\begin{definition}
Let $\lambda$, $\kappa$, and $\nu$ be infinite cardinals.
A family $\mc{F} \subseteq {^\lambda \kappa}$
 is said to be $\nu$-independent iff
 $$(\forall F \in [\mc{F}]^{<\nu})
 (\forall \varphi: F \to \kappa)
 (\exists x \in \lambda)
 (\forall f \in F)\,
 f(x) = \varphi(f).$$
\end{definition}

We will now recall an old result.
For the sake of this paragraph,
 let $I(\lambda, \kappa, \nu, \mu)$
 be the statement ``there exists a
 $\nu$-independent family $\mc{F} \subseteq {^\lambda \kappa}$
 of size $\mu$''.
$I(\omega, 2, \omega, 2^\omega)$ and
$I(2^\omega, 2, \omega, 2^{2^\omega})$ were both shown
 in \cite{FK}.
For arbitrary infinite $\lambda$,
 $I(\lambda, 2, \omega, 2^\lambda)$ was shown
 in \cite{Hausdorff}.
For infinite cardinals $\lambda$ and $\kappa$ such that
 $2^{<\kappa} \le \lambda$,
 $I(\lambda, 2, \kappa, 2^\lambda)$ was shown
 in \cite{Tarski}.
Finally,
 for infinite cardinals $\lambda$ and $\kappa$ such that
 $\lambda^{<\kappa} = \lambda$,
 $I(\lambda, \lambda, \kappa, 2^\lambda)$ was shown
 in \cite{EK}.
We state this last result as the theorem below.
For a proof of this theorem, see
 (a) $\Rightarrow$ (d) of Theorem 3.16 in \cite{CN}.
See also the end of Chapter 3 in \cite{CN}
 for more information.

\begin{theorem}
\label{ind_family}
If $\lambda^\kappa = \lambda$,
 then there is a
 $\kappa^+$-independent family
 of $2^\lambda$ functions from $\lambda$ to $\kappa$.
More generally, if
 $\lambda^{<\kappa} = \lambda$,
 then there is a
 $\kappa$-independent family
 of $2^\lambda$ functions from $\lambda$ to $\kappa$.
\end{theorem}

Immediately, this gives us the desired bound:

\begin{cor}
\label{main_cor}
If $\lambda^\kappa = \lambda$,
 then $\cf({^\lambda \kappa}) = 2^\lambda$.
\end{cor}
\begin{proof}
Assuming $\lambda^\kappa = \lambda$,
 let $\mc{F}$ be a
 $\kappa^+$-independent
 family of $2^\lambda$ functions from $\lambda$ to $\kappa$.
Since this family is $\kappa^+$-independent,
 every size $\kappa$ subset is unbounded.
Suppose, towards a contradiction,
 that there is some size $< 2^\lambda$
 family $\mc{D} \subseteq {^\lambda \kappa}$
 that is cofinal in ${^\lambda \kappa}$.
Note that $\kappa < 2^\lambda$.
By the pigeon hole principle,
 there is some $g \in \mc{D}$ which dominates
 at least $\kappa$ members of $\mc{F}$.
This contradicts every size $\kappa$
 subset of $\mc{F}$ being unbounded.
\end{proof}

This result was probably known, but the author
 could find no reference for it.
An equivalent and amusing way to write this
 result is as follows:
 $$\cf({^{(\lambda^\kappa)} \kappa}) = 2^{\lambda^\kappa}.$$

As a special case of this corollary,
 we have the following:
\begin{cor}
The cofinality of the set of all functions
 from $\mbb{R}$ to $\omega$ ordered by everywhere
 domination is $2^{2^{\omega}}$.
\end{cor}

The last corollary
 gives us a different proof
 of a well-known result
 (which is attributed to Kunen
 in \cite{Jech1}):
\begin{cor}
$\ch$ implies
 $\cf({^{\omega_1}}\omega) = 2^{\omega_1}$.
\end{cor}

There is a more general
 consequence of the theorem above
 which can be stated using the language
 of challenge response relations
 (see section 4 of \cite{Blass1}):
\begin{proposition}
\label{conjunct_prop}
Suppose $\mc{R} = \langle R_-, R_+,R \rangle$
 is a challenge response relation.
Suppose $\kappa = ||\mc{R}^\perp||$.
Let $\lambda$ be a cardinal such that
 $\lambda^\kappa = \lambda$.
Let $\tilde{\mc{R}} :=
 \langle {^\lambda R_-}, {^\lambda R_+}, \tilde{R} \rangle$
 be the conjunction of $\mc{R}$ with itself $\lambda$
 many times.
That is,
 $f \tilde{R} g$ iff $(\forall x \in \lambda)\, f(x) R g(x)$.
Then $||\tilde{\mathcal{R}}|| = 2^\lambda$.
In fact, there is a set
 $\mc{F} \subseteq {^\lambda R_-}$ of size $2^\lambda$ such that
 for every size $\kappa$ subset $A'$ of $\mc{F}$,
 there is no
 $g \in {^\lambda R_+}$ such that
 $(\forall f \in A')\, f \tilde{R} g$.
\end{proposition}

\begin{proof}
Let $A = \{ a_\alpha : \alpha < \kappa \}
 \subseteq R_-$ be a set of size $\kappa$
 such that there is no single $b \in R_+$
 such that $(\forall \alpha < \kappa)\, a_\alpha R b$.
Using Theorem~\ref{ind_family},
 we obtain a set
 $\mc{F} = \{ f_\beta : \beta < 2^\lambda\}
 \subseteq {^\lambda R_-}$ of size $2^\lambda$
 such that for every injection
 $i : \kappa \to 2^\lambda$,
 there exists an $x \in \lambda$ such that
 $$(\forall \alpha < \kappa)\,
 f_{i(\alpha)}(x) = a_\alpha.$$
The set $\mc{F}$ is as desired.
\end{proof}

\section{The Main Theorem}

Let $\kappa$ and $\lambda$ be infinite cardinals.
The theorem below can be remembered as
 ``for every $A \subseteq \lambda$ there is a function
 $f$ from ${^\kappa \lambda}$ to $\kappa$ such that
 $A$ is constructible from ${^\kappa \lambda}$ and any
 $g$ which dominates $f$''.

\begin{maintheorem}
\label{maintheorem}
For every $A \subseteq \lambda$ there is a
 function $f: {^\kappa \lambda} \to \kappa$
 such that whenever $M$ is a
 transitive model of $\zf$ such that
 ${^\kappa \lambda} \subseteq M$
 and some $g : {^\kappa \lambda} \to \kappa$ in $M$
 dominates $f$,
 then $A \in M$.
\end{maintheorem}
\begin{proof}
Fix $A \subseteq \lambda$.
Define $f$ by
 $$f(x) := \begin{cases}
 0 & \mbox{if } (\forall \alpha < \kappa)\, x(\alpha) \not\in A, \\
 \alpha + 1 &\mbox{if } x(\alpha) \in A \mbox{ but }
 (\forall \beta < \alpha)\, x(\beta) \not\in A.
 \end{cases}$$
Let $M$ be a transitive model of $\zf$
 such that ${^\kappa \lambda} \subseteq M$ and $A \not \in M$.
Suppose, towards a contradiction, that there is some
 $g \in M$ that dominates $f$.
Let $B$ be the set
 $$B := \{ t \in {^{<\kappa}\lambda} :
 g(x) \ge \dom(t) \mbox{ for all } x \mbox{ extending } t\}.$$
Notice that $B \in M$.

For all $a \in \lambda$,
 $a \in A$ implies $\langle a \rangle \in B$.
Thus, there must be some $a_0 \in \lambda$ such that
 $a_0 \not\in A$ but $\langle a_0 \rangle \in B$.
If there was not, then $A$ could be defined in $M$ by
 $A = \{ a \in \lambda : \langle a \rangle \in B\}$,
 which would contradict the fact that $A \not\in M$.

Next, for all $a \in \lambda$,
 $a \in A$ implies $\langle a_0, a \rangle \in B$.
Thus, by similar reasoning as before,
 there must be some $a_1 \in \lambda$ such that
 $a_1 \not\in A$ but $\langle a_0, a_1 \rangle \in B$.
Continuing like this, we can construct a sequence
 $x \in {{^\kappa}\lambda}$ such that
 $(\forall \alpha < \kappa)\, x \restriction \alpha \in B$.
This means that $(\forall \alpha < \kappa)\, g(x) \ge \alpha$.
This contradicts $g$ being well-defined at $x$.
\end{proof}

For demonstration purposes,
 we show how this also
 implies Corollary~\ref{main_cor}:

\begin{cor}
If $\lambda^\kappa = \lambda$, then
 $\cf( {^\lambda \kappa} ) = 2^\lambda$.
\end{cor}
\begin{proof}

Since there is a bijection between
 $\lambda$ and ${^\kappa \lambda}$,
 it suffices to show that the cofinality of the set
 of all functions from ${^\kappa \lambda}$ to $\kappa$
 is at least $2^\lambda$.
Consider an arbitrary family $\mc{A}$ of functions from
 ${^\kappa \lambda}$ to $\kappa$ of size $< 2^\lambda$.
We will show that it is not dominating.

Let $$Z := \mc{A} \cup {^{\kappa}\lambda} \cup \lambda.$$
Let $E \prec V$\footnote{Instead of using $V$, we could use
 $H(\theta)$ for some appropriate regular cardinal $\theta$.}
 be such that $Z \subseteq E$ and
 $|E| < 2^\lambda$.
Such an $E$ exists because
 $$|Z| = \max\{ |\mc{A}|, \lambda \} < 2^\lambda.$$
Let $M := \pi(E)$ be the transitive collapse of $E$.
Since ${^{\kappa}\lambda} \subseteq E$, we have
 $\pi(g) = g$ for all $g \in \mc{A}$.
Hence,
 $\mc{A} \subseteq M$.
Since $|M| < 2^\lambda$,
 there is some $A \in \p(\lambda) \m M$.
We may now apply the main lemma to get that
 there is some $f$ not dominated by any member of $M$.
In particular, such an $f$ is not dominated by any
 member of $\mc{A}$.
This completes the proof.
\end{proof}





Notice that the main theorem
 required ${^\kappa \lambda} \subseteq M$.
That is, $M$ contains the domains of
 the functions involved.
Dropping the requirement not only weakens
 the conclusion in the obvious way,
 but it also forces us to consider the possibility
 that the set $B$ has no length $\kappa$ branch in $M$,
 which breaks the proof.
We can remove the ${^\kappa \lambda} \subseteq M$ assumption
 if we make various special modifications, and we will
 present several of them.
In the case that $\kappa = \omega$,
 we can use
 the fact that well-foundedness of trees is absolute.
This next theorem can be remembered as
 ``for every $A \subseteq \lambda$ there is a function
  $f$ from ${^\omega \lambda}$ to $\omega$ such that $A$
  is constructible from $\lambda$ and any $g$
  which dominates $f$''.
\begin{theorem}
\label{omega_theorem}
For every $A \subseteq \lambda$ there is a function
 $f: {^\omega \lambda} \to \omega$ such that whenever
 $M$ is a transitive model of $\zf$ such that
 $\lambda \in M$ and some
 $g : ({^\omega \lambda})^M \to \omega$ in $M$
 satisfies
 $$(\forall x \in ({^\omega \lambda})^M) f(x) \le g(x),$$
 then $A \in M$.
\end{theorem}
\begin{proof}
Fix $A \subseteq \lambda$.
Define $f$ as is done in the main theorem.
Let $M$ be a transitive model of $\zf$
 such that $\lambda \in M$ and which contains
 some appropriate $g$.
Assume, towards a contradiction, that $A \not\in M$.
Define the set $B \subseteq {^{<\omega}\lambda}$
 as is done in the proof of Theorem~\ref{maintheorem},
 except that the $x$ in the definition ranges
 over elements of $(^\omega \lambda)^M$ extending $t$.
We have $B \in M$.
Let $T \subseteq {^{<\omega} \lambda}$ be those
 elements of $B$ all of whose initial segments
 are also in $B$.
Note that $T \in M$.
We may repeat the proof of the main lemma to get
 that $T$ has a path in $V$.
Since well-foundedness is absolute,
 there is some $x' \in [T] \cap M$.
This $x'$ witnesses that
 $(\forall n \in \omega)\, g(x') \ge n$,
 which is a contradiction.

\end{proof}

If $\kappa > \omega$,
 there does not seem to be an obvious way to extend the
 last lemma, even if $\kappa$ has some
 large cardinal property.
However, if we are willing to sacrifice the sharpness of
 ``a function $f$ such that all dominators of $f$ can
 \textit{construct} $A$'', then we can use an elementary
 substructure argument to fill in the part of the proof
 where we need $B$ to have a length $\kappa$ path in $M$.
We also need the substructure $M$ to include ${^{<\kappa} \lambda}$
 for technical reasons which seem unavoidable.
We state the next theorem but do not prove it,
 as it is very similar to the ones above.

\begin{theorem}
For every $A \subseteq \lambda$ there is a function
 $f : {^\kappa \lambda \to \kappa}$ such that
 whenever $\langle M, \in \rangle \prec V$ is such that
 ${^{<\kappa}\lambda} \subseteq M$ and some
 $g : {^\kappa \lambda} \to \kappa$ in $M$ satisfies
 $$(\forall x \in {^\kappa \lambda})\, f(x) \le g(x),$$
 then $A \in M$.
\end{theorem}

If we assume that $\kappa$ is weakly compact in $M$,
 then we can build a function
 from ${^\kappa 2}$ to $\kappa$ that can only
 be dominated by a function in $M$ if
 $A \in M$:

\begin{theorem}
For every $a \in {^\kappa 2}$ there is a function
 $f : {^\kappa 2} \to \kappa$ such that
 whenever $M$ is a transitive model of $\zf$ such that
 $\kappa \in M$,
 ${^{<\kappa} 2} \subseteq M$,
 $(\kappa$ is weakly compact$)^M$,
 and some $g : ({^\kappa 2})^M \to \kappa$ in $M$ satisfies
 $$(\forall x \in ({^\kappa 2})^M)\, f(x) \le g(x),$$
 then $a \in M$.
\end{theorem}
\begin{proof}
Fix $\kappa$ and $a \in {^\kappa 2}$.
Let $f$ be the function
 $$f(x) := \begin{cases}
 \alpha & \mbox{if } x(\alpha) \not= a(\alpha) \mbox{ but }
  (\forall \beta < \alpha)\, x(\beta) = a(\beta),\\
 0 & \mbox{if } x = a.
 \end{cases}$$
Let $M$ and $g$ be as in the statement
 of the theorem.
Assume, towards a contradiction, that
 $a \not\in M$.
Define $B$ similarly as before:
 $$B := \{ t \in {^{<\kappa} 2} : g(x) \ge \dom(t)
 \mbox{ for all } x \in (^{\kappa} 2)^M
 \mbox{ extending } t\}.$$
We have $B \in M$.
Just as in Theorem~\ref{omega_theorem},
 define $T \subseteq {^{< \kappa} 2}$ to be 
 the set of elements of $B$ all of whose initial segments
 are also in $B$.
We have $T \in M$.
Note also that
 since $(\kappa$ is strongly inaccessible$)^M$,
 we have $(T$ is a $\kappa$-tree$)^M$.

Now, since $a \not\in M$,
 one can see that $a \restriction \alpha$
 is in $M$ for each $\alpha < \kappa$.
In particular, $T$ has height $\kappa$.
This is calculated in $V$,
 but it is clearly absolute, so
 $(T$ has height $\kappa)^M$.
Since $(T$ is a $\kappa$-tree$)^M$,
 there must be some
 $x' \in [T] \cap M$.
As in Theorem~\ref{omega_theorem},
 this $x'$ witnesses that
 $(\forall \alpha < \kappa)\, g(x') \ge \alpha$,
 which is a contradiction.
\end{proof}

\section{A Consequence}

We can use the main theorem of the last section
 to obtain a result that is incomparable to
 Proposition~\ref{conjunct_prop}.
In the language of \cite{Blass1},
 this proposition is saying there exists a
 morphism between two relations.
We get this result because the main
 theorem of last section is really saying that
 a morphism exists between a certain domination
 relation and the constructibility relation:

\begin{proposition}
Let $\langle P, \le_P \rangle$
 be a partial ordering and let
 $\kappa$ be the smallest size of
 an unbounded subset of $P$
 ($\kappa = \mf{b}\langle P, \le_P \rangle$).
Let $\lambda$ be a cardinal and assume
 $|P| \le 2^\lambda$.
Let $\mc{F} :=
 \langle {^\lambda P}, \le_{^\lambda P} \rangle$
 be the partial ordering of all functions
 from $\lambda$ to $P$ given by
 $f \le_{^\lambda P} g$ iff
 $$(\forall \alpha < \lambda) f(\alpha) \le_P g(\alpha).$$
Assume also that $\lambda^\kappa = \lambda$.
Then $\cf(\mc{F}) = 2^\lambda$.
Moreover, there is a function
 $\phi^- : {^\lambda \kappa} \to {^\lambda P}$
 and there is a function
 $\phi^+ : {^\lambda P} \to {^\lambda \kappa}$
 such that
 $$(\forall g \in {^\lambda \kappa})
 (\forall f \in {^\lambda P})\,
 \phi^-(g) \le_{^\lambda P} f \Rightarrow g \le \phi^+(f).$$
\end{proposition}
\begin{proof}
Since $|P| \le 2^\lambda$, we have
 $|{^\lambda P}| \le (2^\lambda)^\lambda = 2^\lambda$, so
 $\cf(\mc{F}) \le 2^\lambda$.
By Corollary~\ref{main_cor},
 since $\lambda^\kappa = \lambda$, we have
 $\cf( {^\lambda \kappa} ) = 2^\lambda$.
Once we define the appropriate functions
 $\phi^-$ and $\phi^+$,
 it will follow that
 $\cf(^\lambda \kappa) \le \cf(\mc{F})$,
 and so $\cf(\mc{F}) = 2^\lambda$.

Using induction, we can construct
 an unbounded chain $\langle a_\alpha : \alpha < \kappa \rangle$
 in $\langle P, \le_P \rangle$.
That is, $(\forall \alpha < \beta < \kappa)\,
 a_\alpha \le_P a_\beta$
 and there is no $a \in P$ such that
 $(\forall \alpha < \kappa)\, a_\alpha \le_P a$.
Let $\phi^+ : {^\lambda P} \to {^\lambda \kappa}$
 be the function
 $$\phi^+(f) := (x \mapsto
 \min\{ \alpha < \kappa : a_\alpha \not\le_P f(x) \}).$$
Let $\phi^- : {^\lambda \kappa} \to {^\lambda P}$
 be the function
 $$\phi^-(g) := (x \mapsto a_{g(x)}).$$
These functions are as desired.
\end{proof}

As an example of how to use this proposition,
 let $P$ be the set of Lebesgue measure
 zero Borel subsets of $\mbb{R}$.
Let $A \le_P B$ iff $A \subseteq B$.
A straightforward argument shows
 that the smallest size $\kappa$
 of an unbounded subset of
 $P$ is $\add(\mc{L})$
 (the additivity of Lebesgue measure,
 which appears in Cicho\'n's Diagram).
Let $\lambda = 2^\omega$.
Since every element of $P$ is Borel,
 $|P| = \lambda$, but we need only that
 $|P| \le 2^\lambda$.
If $\lambda^\kappa = \lambda$,
 then $\cf(\mc{F}) = 2^\lambda$.
That is,
 if $2^{\add(\mc{L})} = 2^\omega$,
 then $\cf({\mc{F}}) = 2^{2^\omega}$.

\section{Interpretation of Main Theorem}

Within this section,
 whenever $B_1, ..., B_n$ are sets,
 let $M(B_1, ..., B_n)$ refer to the smallest
 transitive model of $\zf$ which contains
 $B_1, ..., B_n$ as elements.
This is well-defined, although
 it could be a proper class.
To say that $A \in M(B_1, ..., B_n)$
 is to say that $A$ is constructible
 from $B_1, ..., B_n$ in a certain sense.

The main theorem can now been seen as
 the following statement:
 $$(\forall A \subseteq \lambda)
 (\exists f: {^\kappa \lambda} \to \kappa)
 (\forall g: {^\kappa \lambda} \to \kappa)
 [ f \le g \Rightarrow A \in M(g, {^\kappa \lambda}) ].$$
In the language of \cite{Blass1},
 this is saying there is a morphism from the
 domination relation of functions from
 ${^\kappa \lambda}$ to $\kappa$ to
 the relation $R$ defined by
 $A R g :\Leftrightarrow A \in M(g, {^\kappa \lambda})$.

This is analogous to the situation
 with functions from $\omega$ to $\omega$.
There, however, we have only
 that for every \textit{hyperarithmetical}
 $A \subseteq \omega$,
 there is a function $f: \omega \to \omega$ such that
 if a function $g: \omega \to \omega$ everywhere dominates $f$,
 then $A$ is Turing reducible to $g$.
Of course, changing ``everywhere dominates'' to
 ``eventually dominates''
 makes no difference, because making finite modifications
 to a function does not change its Turing degree.
See \cite{Blass2},
 \cite{Solovay}, and
 \cite{Jockusch} for more details.

\end{document}